\documentclass{amsart}

\usepackage{amsmath}
\usepackage{amsthm}
\usepackage{amssymb}
\usepackage{amsbsy}
\usepackage{amsfonts}
\usepackage{amstext}
\usepackage{amscd}
\usepackage{tikz}
\usepackage{graphicx}

\usepackage{color}

\numberwithin{equation}{section}
\theoremstyle{plain}
\newtheorem{thm}{Theorem}[section]
\newtheorem{prop}[thm]{Proposition}

\newtheorem{lem}[thm]{Lemma}
\theoremstyle{definition}

\newtheorem{rem}[thm]{Remark}
\newtheorem{defi}[thm]{Definition}

\newcommand\blfootnote[1]{%
  \begingroup
  \renewcommand\thefootnote{}\footnote{#1}%
  \addtocounter{footnote}{-1}%
  \endgroup
}

\usepackage{graphicx}
\usepackage{amsmath,amsthm,amsfonts}

\usepackage{color}

\begin{document}
\title[Berry-Esseen for finite free convolution]
{A Berry-Esseen type theorem for finite free convolution}

\author{Octavio Arizmendi}
\address{Centro de Investigaci\'on en Matem\'aticas, A.C., Jalisco S/N, Col. Valenciana CP: 36023 Guanajuato, Gto, M\'exico}

\author{Daniel Perales}

\begin{abstract}
We prove that the rate of convergence for the central limit theorem in finite free convolution is of order $n^{-1/2}$. 

\end{abstract}
\maketitle

\blfootnote{\begin{minipage}[l]{0.3\textwidth} \includegraphics[trim=10cm 6cm 10cm 5cm,clip,scale=0.15]{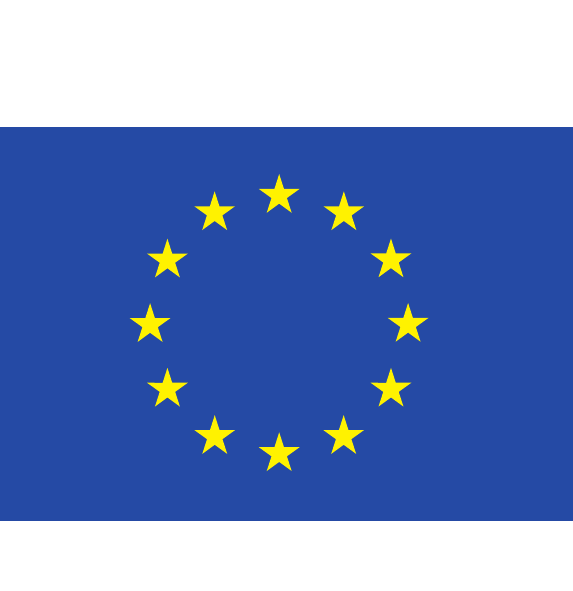} \end{minipage} 
 \hspace{-2cm} \begin{minipage}[l][1cm]{0.82\textwidth}
 	  This project has received funding from the European Union's Horizon 2020 research and innovation programme under the Marie Sk\l{}odowska-Curie grant agreement No 734922.
 	\end{minipage}}

\section{Introduction}

In recent years, the convolution of polynomials, first studied by Walsh \cite{walsh1922location},  was revisited by Marcus, Spielman and Srivastava \cite{marcus2015finite}, in order to exhibit bounds for the eigenvalues of expected characteristic polynomials of certain $d$-regular graphs, in their aim to construct bipartite Ramanujan graphs of all sizes \cite{marcus2015interlacing4}. The authors refer to this convolution as finite free additive convolution because of its relation to free convolution, see  \cite{arizmendi2018cumulants,marcus2015polynomial,marcus2015interlacing4}. 

In \cite{marcus2015polynomial}, Marcus showed that the Central Limit Theorem for this convolution is given by the polynomial $D_{1/\sqrt{d}} (H_d),$ where $H_{d}$ is an Hermite polynomial and may be written as
$$H_{d}(x)=d!\sum _{i=0}^{\lfloor {\tfrac {d}{2}}\rfloor }{\frac {(-1)^{i}}{i!(d-2i)!}}{\frac {x^{d-2i}}{2^{i}}},$$
and in general for any $\lambda>0$ and polynomial $p$ of degree $d$, $D_\lambda p(x):=\lambda^{d} p(x/\lambda)$ is the \emph{dilation by $\lambda$ of the polynomial $p$}.

In this note we are interested in the rate of convergence in the Central Limit Theorem for finite free convolution. Recall that for the central limit theorem in probability, the Berry-Esseen Theorem \cite{Ber,Es} states that if $\mu$ is a probability measure with $m_1(\mu)=0$, $m_2(\mu)=1$, and $\int_{\mathbb{R}} | x |^3 d\mu < \infty$, then the distance to the standard Gaussian distribution $\mathcal{N}$ is bounded as follows
$$d_{kol}(D_{\frac{1}{\sqrt{n}}}\mu^{\ast n},\mathcal{N})\leq C \frac{\int_{\mathbb{R}} | x |^3 d\mu}{\sqrt{n}},$$
where $d_{kol}$ denotes the Kolmogorov distance between measures, $D_b\mu$ denotes the dilation of a measure $\mu$ by a factor $b>0$,  $\ast$ denotes the classical convolution, and $C$ is an absolute constant.
Our main reasult shows that as for the classical and free case \cite{chistyakov2008limit,kargin2007berry}, in finite free probability we also achieve a rate of order $n^{-1/2}$. However, we use the L\'evy distance (see Section \ref{sec:Levy} for more details) instead of Kolmogorov distance, the reason being that we are dealing with measures supported in $d$ atoms with size $1/d$ and thus we cannot expect better.

Thus, for two polynomials of degree $d$, $p$ and $q$, let us define the distance between them to be $L(p,q):=d_L(\mu_p,\mu_q),$
where $d_L$ is the L\'evy distance and the measures $\mu_p$ and $\mu_q$ are the distributions of $p$ and $q$, respectively. 

In this language we can state our contribution as follows.
\begin{thm}
Let $d\in\mathbb{N}$ and let $p$ be a real polynomial of degree $d$ such that the first two moments of $\mu_p$ are $m_1=0$ and $m_2=1$. Then, there exists an \textbf{absolute constant} $C_d$, only depending on $d$, such that for all $n>0$,
$$L\left(D_{1/\sqrt{n}}(p^{\boxplus_d n}), D_{1/\sqrt{d}} (H_d)\right)< \frac{C_d}{\sqrt{n}},$$
\end{thm}
   
  The main tool to prove the above rate of convergence are the cumulants for finite free convolution, as we defined in \cite{arizmendi2018cumulants}.  These cumulants give a combinatorial approach to investigate this convolution and its relation to free probability. In particular we showed that finite free cumulants approach free cumulants, providing a
combinatorial perspective to the fact that finite free convolution approaches free convolution in the limit. Using these cumulants we were able to show that some properties of free convolution are valid already in the finite free case. The above theorem is another instance of the fact that many properties in free probability already appear in the finite level.

 Apart from this introduction this note consists of two sections. Section 2 gives the preliminaries for the theory of finite free probability and in Section 3 we give the proof of the main theorem, Theorem 1.1.
 
\section{Preliminaries}

We give very basic preliminaries on finite free convolution we refer to \cite{arizmendi2018cumulants,marcus2015polynomial} for details. 

\subsection{Finite Free Convolution}
 For two polynomials,
  $p(x) = \sum_{i=0}^d x^{d-i} (-1)^i  a^p_i$ and  $q(x) = \sum_{i=0}^d x^{d-i} (-1)^i  a^q_i$, the finite free additive convolution of $p$ and $q$ is given by
\begin{equation*}  p (x) \boxplus_{d} q (x) = \sum_{k=0}^d x^{d-r} (-1)^r  \sum_{i+j=r} \frac{(d-i)!(d-j)!}{d! (d-i-j)!} a^p_i a^q_j.
\end{equation*}

The finite $R$-transform of a polynomial  is defined by
\begin{equation} \label{definition R}
\mathcal{R} ^d _{p} (s) \equiv - \frac{1}{d} \frac{\partial}{\partial s} \ln \left( \sum_{i=0}^d  \frac{(-d)^i  a^p_i}{(d)_i} s^{i} \right) \qquad \text{mod } [s^{d}],
\end{equation} when $p$ is  the monic polynomial $p(x) = \sum_{i=0}^d x^{d-i} (-1)^i  a^p_i$.

We consider the truncated $R$-transform given by the sum of the first $d$ terms in the series expansion of $\mathcal{R} ^d _{p} $, which will have the cumulants as coefficients.

\begin{defi}[\cite{arizmendi2018cumulants}]
\label{defcumfin}
Let $p$ be a monic polynomial of degree $d$, and suppose the $\mathcal{R}^d_{p} (s)$ satisfies  
$$\mathcal{R}^d_{p} (s) \equiv \sum_{j=0}^{d-1} \kappa_{j+1}(p) s^j \quad \text{ mod } [ s^d ].$$

\begin{enumerate}
\item We call the sum of the first $d$ terms in the series expansion of $\mathcal{R} ^d$ the \emph{truncated $R$-transform} and denote by $\mathcal{\tilde{ R}}^d_{p} (s)$, i.e. $$\mathcal{\tilde{ R}}^d_{p} (s) :=\sum_{j=0}^{d-1} \kappa_{j+1}(p) s^j. $$

\item The numbers $\kappa_1(p), \kappa_2(p), \dots , \kappa_{d}(p)$ will be called the finite free cumulants. To simplify notation we will omit the dependence on $p$ when we deal with only one polynomial.
\end{enumerate}
\end{defi} 

We want to use the combinatorial framework in terms of moments for these cumulants. Hence, for a polynomial $p$ with roots $\lambda_1,....,\lambda_n$ we define the moments of $p$, by the formula $m_n(p)=\frac{1}{d}\sum^d_{i=1} \lambda_i^n$.

 These finite free cumulants satisfy the following properties which are the analog of  the properties in the axiomatization of cumulants by Lehner \cite{lehner2002free}, in non-commutative probability. 

\begin{enumerate}
\item[(1)] \textbf{Polynomial in the first $n$ moments:} $k_n(p)$ is a polynomial in the first $n$ moments of $p$ with leading term $$\frac{d^n}{(d)_n} m_n(p).$$

\item[(2)] \textbf{Homogeneity:} for all monic polynomials $p$ and $\lambda\neq 0$ we have $$\kappa_n(D_\lambda (p)) = \lambda^n \kappa_n(p).$$

\item[(3)] \textbf{Additivity:} for all monic polynomials $p$ and $q$, we have $$\kappa_n(p\boxplus_d q) = \kappa_n(p)+\kappa_n(q).$$
\end{enumerate}

\subsection{Moment-cumulant formula} Moment-cumulant formulas involve summing over partitions on the set $[n].$
Let us introduce this definition and some notation.
\begin{defi}
We call $\pi =\{V_{1},...,V_{r}\}$ a \textbf{partition }of the set $[n]:=\{1, 2,\dots, n\}$
if $V_{i}$ $(1\leq i\leq r)$ are pairwise disjoint, non-void
subsets of $[n]$, such that $V_{1}\cup V_{2}...\cup V_{r}=\{1, 2,\dots, n\}$. We call $
V_{1},V_{2},\dots,V_{r}$ the \textbf{blocks} of $\pi $. The number of blocks of 
$\pi $ is denoted by $\left\vert \pi \right\vert $. We will denote the set of partitions  of $[n]$ by $\mathcal{P}(n)$.
\end{defi}

The set $\mathcal{P}(n)$ can be equipped with the partial order $\leq$ of reverse refinement ($\pi\leq\sigma$ if and only if every block of $\pi$ is completely contained in a block of $\sigma$). With this order the minimum is given by the partition with $n$ blocks, $0_n=\{\{1\},\{2\},\cdots, \{n\}\}$, and the maximum is given by the partition with $1$ block, $1_n=\{\{1,2,\cdots,n\}\}$. 

Thus one can consider the incidence algebra of $\mathcal{P}(n)$. For two partitions $\sigma,\rho $ in the set of partitions $\mathcal{P}(n)$ the M\"obius function is given by 
$$\mu (\sigma ,\rho)=(-1)^{{|\sigma|-|\rho|}}(2!)^{{r_{3}}}(3!)^{{r_{4}}}\cdots ((n-1)!)^{{r_{n}}},$$
where $r_i$ is the number of blocks of $\rho$ that contain exactly $i$ blocks of $\sigma$.
In particular, for $\sigma=0_n$ we have
\begin{equation*}
\mu (0_n ,\rho )=(-1)^{n-{|\rho|}}(2!)^{{t_{3}}}(3!)^{{t_{4}}}\cdots ((n-1)!)^{{t_{n}}},
\end{equation*}
where $t_i$ is the number of blocks of $\rho$ of size $i$.

Given a sequence of complex numbers $f=\{f_n\}_{n\in \mathbb{N}}$ we may extend $f$ to partitions in a multiplicative way by the formula $$f_\pi=f_{|V_1|}f_{|V_2|}\cdots f_{|V_n|},$$
where $V_1,\dots,V_n$ are the blocks of $\pi.$ In this note we will frequently use the multiplicative extensions of the Pochhammer sequence $(d)_n=(d)(d-1)\cdots (d-n+1)$ and the factorial sequence $n!$, whose extensions will be denoted by $(d)_\pi$  and $N!_\pi$, respectively.

In \cite{arizmendi2018cumulants}, we gave formulas that relate the moments and coefficients of a polynomial with its finite free cumulants. First, we have a formula that writes coefficients in terms of cumulants.
\begin{prop}[Coefficient-cumulant formula] \label{aaacum}
Let $p(x) = \sum_{i=0}^d x^{d-i} (-1)^i  a_i$ be a polynomial of degree $d$ and let $(\kappa_n)^d_{n=1}$ be its finite free cumulants. The following formulas hold.
\begin{equation} \label{aaatercum}
a_{n} = \frac{(d)_n}{d^nn!} \sum _{\pi \in \mathcal{P}(n)} d ^{| \pi |}  \mu(0_n,\pi) \kappa_{\pi}, \qquad n\in\mathbb{N}.
\end{equation}
\end{prop}

We also have a moment-cumulant formula for finite free cumulants:
\begin{prop}
Let $p$ be a monic polynomial of degree $d$ and let  $(m_n)^\infty_{n=1}$ and $(\kappa_n)^d_{n=1}$, be the moments and cumulants of $p$, respectively. Then
\begin{align}
\kappa_n &= \frac{(-d)^{n-1}}{(n-1)!} \sum_{\sigma \in \mathcal{P}(n)} d^{|\sigma|}\mu(0,\sigma) m_{\sigma}  \sum_{\pi \geq \sigma} \frac{\mu(\pi,1_n)}{(d)_{\pi}},\nonumber
\end{align}
for $n=1,\ldots, d$ and
\begin{align} 
m_n   &= \frac{(-1)^n}{d^{n+1}(n-1)!} \sum_{\sigma \in \mathcal{P}(n)}d^{|\sigma|}\mu(0,\sigma)\kappa_\sigma \sum_{\pi \geq \sigma}- \mu(\pi,1_n)  (d)_\pi, \nonumber
\end{align}
for $n\in\mathbb{N}$.
\end{prop}

\begin{rem}
The explicit moment-cumulant formulas of the first three finite cumulants are
\begin{eqnarray*}
\kappa_1&=& m_1,\qquad
\kappa_2=\frac{d}{d-1}(m_2 - m_1^2),\\
\kappa_3&=&\frac{d^2}{(d-1)(d-2)}(2m_1^3 - 3m_1m_2 + m_3),
\end{eqnarray*}
and the explicit moment-cumulant formulas of the first three finite moments are
\begin{eqnarray*}
m_1& = &\kappa_1,\qquad
m_2 = \frac{d-1}{d}\kappa_2 + \kappa_1^2,\\
m_3& = &\frac{(d-1)(d-2)}{d^2}\kappa_3+\frac{3(d-1)}{d}\kappa_2\kappa_1 +\kappa_1^3.
\end{eqnarray*}
\end{rem}

\subsection{Convergence of polynomials and L\'evy distance}
\label{sec:Levy}
In this setting of \cite{arizmendi2018cumulants,marcus2015polynomial} convergence of polynomials is pointwise convergence of the coefficients. We prefer to consider the weak convergence of the induced measures since it is common with the free probability setting.  Thus,  for a polynomial $p$, with roots $\lambda_1,\lambda_2,\dots,\lambda_n$, we define its distribution $\mu_p$ as the uniform measure on the roots of $p$, $\mu_p=\tfrac{1}{d}\sum_i\delta_{\lambda_i}$.

To quantify this convergence we use the L\'evy distance
$$d_L(\mu,\nu):=\inf \{ \epsilon>0 \: | \: F(x-\epsilon)-\epsilon\leq G(x)\leq F(x+\epsilon)+\epsilon \: \: \: \text{for all } x\in \mathbb{R} \},$$
where $F$ and $G$ are the cumulative distribution functions of $\mu$ and $\nu$ respectively.

\section{Proof of Theorem 1.1}

Before going in to the proof of the main theorem we prove a couple of lemmas about the support and cumulants of polynomials with mean $0$ and variance $1$. 

\begin{lem}
Let $p$ be a real polynomial of degree $d$ with $\kappa_1=0$ and $\kappa_2=1$. Then the support of $p$ is contained in $(-\sqrt{d-1},\sqrt{d-1})$.
\end{lem}
\begin{proof}
If $\kappa_1=0$ and $\kappa_2=1$ then $$1=\kappa_2=\frac{d}{d-1}m_2=\frac{1}{d-1}\sum_{i=1}^d\lambda_i^2.$$  This means that $\lambda_i^2<d-1$ (strict because there is at least another non-zero $\lambda$) and thus $|\lambda_i|<\sqrt{d-1}$ for all $i=1,\dots,d$.
\end{proof}

\begin{lem}
Let $p$ be a real polynomial of degree $d$ with $\kappa_1=0$ and $\kappa_2=1$. Then there exists a constant $c_d$, depending only on $d$, such that 
$\max_{2\leq s\leq d}|\kappa_s(p)|<c_d.$
\end{lem}
\begin{proof}
By the previous lemma $m_n\leq (d-1)^n$ and then $\max_{2\leq s\leq d}|m_s(p)|<(d-1)^d$, so we can bound uniformly $\kappa_n$ by the moment-cumulant formulas. 
\end{proof}
Now we are able to prove the main theorem which we state again for convenience of the reader.

\begin{prop}
Let $p$ be a real polynomial with $\kappa_1=0$ and $\kappa_2=1$. Then, there exists $C_d$ such that for all $n>0$
$$L\left(D_{1/\sqrt{n}}(p^{\boxplus_d n}),D_{1/\sqrt{d}} (H_d)\right)< \frac{C_d}{\sqrt{n}}.$$
\end{prop}

\begin{proof}
Let us denote $h=D_{1/\sqrt{d}} (H_d)$, $p_n=p^{\boxplus_d n}$ and $q_n=D_{1/\sqrt{n}} (p_n)$. By the coefficient-cumulant formula, we know that
\begin{eqnarray*}
a_j^{q_n}&=&\frac{(d)_j}{d^j j!} \sum _{\pi \in \mathcal{P}(j)} d ^{| \pi |}  \mu(0_j,\pi) \kappa_{\pi}(q_n) \\
&=& a_j^{h} + \frac{(d)_j}{d^j j!}\sum_{\pi\in\mathcal{P}(j)\backslash\mathcal{P}_{12}(j) } d^{|\pi|} \mu(0_j,\pi)\kappa_\pi(q_n),
\end{eqnarray*}
where $\mathcal{P}_{12}(j)$ is the set of partitions $\pi=(V_1,\ldots,V_r)\in\mathcal{P}(j)$ such that $|V_i|\leq 2$ for all $i\in\{1,\ldots, r\}$ (i.e., $\pi=(V_1,\ldots,V_r)\in\mathcal{P}(j)\backslash\mathcal{P}_{12}(j)$, if $|V_i|>2$ for some $i\in\{1,\ldots, r\}$).
 
Recall that
$$|\kappa_s(q_n)|=|\kappa_s(D_{1/\sqrt{n}} (p_n))|=\frac{n}{n^{s/2}}|\kappa_s(p)|\leq n^{1-s/2}c,$$
for $s=3,\ldots,d$, where $c:=c_d$ from Lemma 3.2. Thus, for any $3\leq j\leq d$ and $\pi=(V_1,\ldots,V_r)\in\mathcal{P}(j)\backslash\mathcal{P}_{12}(j)$ we get
\begin{equation} \label{mainbound} |\kappa_\pi(q_n)|\leq c^r\cdot n^{r} \cdot n^{-\frac{|V_1|+\cdots+|V_r|}{2}}= c^r n^{r-\frac{j}{2}}\leq c^r n^{\frac{j}{3}-\frac{j}{2}}=c^r n^{-\frac{j}{6}} \leq c^d n^{-\frac{1}{2}}.
\end{equation}

Then,
$$|a_j^{q_n}-a_j^h|\leq \frac{c^d K_1(d)}{\sqrt{n}},\qquad \forall j\in\{1,\ldots,d\}$$
where $$K_1(d)= \max_{1\leq j\leq d} \frac{(d)_j}{d^j j!}\sum_{\pi\in\mathcal{P}(j)\backslash\mathcal{P}_{12}(j) } d^{|\pi|} |\mu(0_j,\pi)|.$$

Let's denote $z_1,z_2,\cdots,z_d$ the $d$ distinct roots of $h$ and $\delta=\frac{1}{2}\min_{1\leq i < j\leq d} |z_i-z_j|$. For $0<\varepsilon<\delta$ we define $B_i=\{z\in\mathbb{C}:|z-z_i|\leq \varepsilon \}$ and $\partial B_i=\{z\in\mathbb{C}:|z-z_i|= \varepsilon \}$. For a fixed root $i$, using the previous bound we can see that for any $z\in \partial B_i$ we have that
$$|q_n(z)-h(z)|\leq \left\vert \sum_{j=0}^d z^{d-j} (-1)^j (a^{q_n}_j-a_j^h) \right\vert \leq \sum_{j=1}^d |z|^{d-m} |a^{q_n}_j-a_j^h|$$
$$\leq \frac{c^d K_1(d)}{\sqrt{n}} \sum_{j=1}^d (|z_i|+|\varepsilon|)^{d-j} \leq \frac{c^d K_1(d)K_2(d)}{\sqrt{n}} $$
where $$K_2(d)=\max_{1\leq i\leq d} \sum_{j=1}^d (|z_i|+|\varepsilon|)^{d-j}.$$

On the other hand, if $z\in \partial B_i$, we know that
$$|h(z)|=|(z-z_0)\cdots (z-z_{n-1})|=|z-z_1|\cdots |z-z_n|\geq |z-z_i| \delta^{d-1}=\varepsilon\delta^{d-1}.$$

Finally, if we take
$$n > \frac{c^{2d} K(d)}{\varepsilon^2} ,$$
where $K(d)=\frac{K^2_1(d)K^2_2(d)}{\delta^{2d-2}}$. Since $c^{2d} K(d)$ does not depend on $i$, we get that for any $i=1,\ldots, n$, if $z\in \partial B_i$, then
$$|q_n(z)-h(z)|\leq \frac{c^d K_1(d)K_2(d)}{\sqrt{n}} < \varepsilon\delta^{d-1}\leq |h(z)|\leq |h(z)|+|q_n(z)|.$$
Thus, Rouch\'e's theorem implies that $q_n$ and $h$ have the same number of roots (counting multiplicity) in $B_i$ for $i=1,\ldots, n$. By the definition of the $B_i$ we know that they are pairwise disjoint and each one contains exactly one of the $d$ roots of $h$. Thus, each $B_i$ contains exactly one of the $d$ roots of $q_n$ implying that distance between the roots of $q_n$ and $h$, (and therefore the L\'evy distance) is less than $\varepsilon$.
\end{proof}

Observe that Theorem 1.1 directly gives a bound for $T$ in the next proposition.
\begin{prop}[\cite{arizmendi2018cumulants}]\label{large t}
Let $p\neq x^d$ be a real polynomial, then there exists $T>0$ such that for all $t>T$ the polynomial $p^{\boxplus_d t}$ has $d$ different real roots. 
\end{prop}

Finally, we show that one cannot do better than $O(\sqrt{n})$ as long as $m_3(p)\neq0$.

\begin{prop}
Let $p$ be a real polynomial with $\kappa_1=0$ and $\kappa_2=1$ and $|m_3|=\alpha \neq0$. Then,  for all $n>0$
$$L\left(D_{1/\sqrt{n}} (p^{\boxplus_d n}),D_{1/\sqrt{d}} (H_d)\right)\geq \frac{\alpha}{3d\sqrt{n}}.$$
\end{prop}

\begin{proof}
We use again the notation $h=D_{1/\sqrt{d}} (H_d)$, $p_n=p^{\boxplus_d n}$ and $q_n=D_{1/\sqrt{n}} (p_n)$ and suppose that $L(q_n,h)< \frac{\alpha}{3d\sqrt{n}}.$
Since $\kappa_1(q_n)=0$, from the moment cumulant formulas we have $m_3(q_n)=\frac{(d-1)(d-2)}{d^2} \kappa_3 (q_n)$ 
and then 
$$|m_3(q_n)|=\tfrac{(d-1)(d-2)}{d^2}|\kappa_3(q_n)|=\tfrac{(d-1)(d-2)}{d^2}\frac{n}{n^{3/2}}|\kappa_3(p)|=\frac{|m_3(p)|}{\sqrt{n}}=\frac{\alpha}{\sqrt{n}}.$$
Since $m_3(h)=0$, we can compute 
$$m_3(q_n)=\frac{1}{d}\sum_{i=1}^d\lambda^3_i(q_n)=\frac{1}{d}\sum_{i=1}^d\lambda^3_i(q_n)-\frac{1}{d}\sum_{i=1}^d\lambda^3_i(h),$$
and thus 
\begin{eqnarray*}|m_3(q_n)|&\leq & \frac{1}{d} \sum_i|\lambda^3_i(q_n)-\lambda^3_i(h)|\\
&=&\frac{1}{d}\sum_{i=1}^d|\lambda_i(q_n)-\lambda_i(h)||\lambda_i^2(q_n)+\lambda_i(q_n)\lambda_i(h)+\lambda_i^2(h)|\\
&<&\frac{1}{d}\sum_{i=1}^d  \left(\frac{\alpha}{3d\sqrt{n}} \right) (d+d+d)= \frac{\alpha}{\sqrt{n}.}
\end{eqnarray*}

Where we used in the last inequality the assumption that $L(q_n,h)< \frac{\alpha}{3d\sqrt{n}}.$ This is a contradiction since the inequality is strict.
\end{proof}

\begin{rem} A specific example with $\kappa_3\neq0$ is the finite free Poisson distribution which has cumulants $\kappa_n=\alpha$ for all $n$. If $\alpha d $ is a positive integer we obtain a valid polynomial.  This is a modification of a Laguerre polynomial, thus we obtain a precise estimate for the distance between the roots of certain Laguerre polynomials and the Hermite polynomials.
\end{rem}

\begin{rem}
A closer look at \eqref{mainbound} shows that if $m_3(p)=0$ then the convergence rate is of order $1/n.$ Indeed, $m_3(p)=0$ implies $\kappa_3(q_n)=\kappa_3(p)=0$. So in \eqref{mainbound} we only need to consider partitions with $|V_i|\geq 4$. In this case, for any $4\leq j\leq d$ we have
$$|\kappa_\pi(p)|\leq c^r n^{r-\frac{j}{2}}\leq c^r n^{\frac{j}{4}-\frac{j}{2}}=c^r n^{-\frac{j}{4}} \leq c^d n^{-1}.$$
\end{rem}

Finally, let us mention that while it is very tempting to let $d$ go to infinity, possibly together with $n$, to obtain a Berry-Esseen bound in free probability, there are two problems. First, the quantity $C_d$, as we obtained it, increases broadly as $d\to \infty$, and second, there is, for the moment not a good bound between finite free and free convolutions.

\bibliographystyle{alpha}

\end{document}